\documentclass[letterpaper,10pt]{article}

\usepackage[margin=0.5in]{geometry}
\usepackage[pagewise]{lineno}
\usepackage[dvipsnames]{xcolor}

\usepackage[hang,small,bf]{caption}
\usepackage{fullpage}
\usepackage{enumerate}
\usepackage{authblk}
\usepackage[colorinlistoftodos]{todonotes}
\usepackage{verbatim}
\usepackage{section, amsthm, textcase, setspace, amssymb, lineno, amsmath, amssymb, amsfonts, latexsym, fancyhdr, longtable, ulem, mathtools}
\usepackage{epsfig, graphicx, pstricks,pst-grad,pst-text,tikz,colortbl}
\usepackage{epsf}
\usepackage{graphicx, color}
\usepackage{float}
\usepackage[rflt]{floatflt}
\usepackage[most]{tcolorbox}
\usepackage{amsfonts}
\usepackage{latexsym,enumitem}
\usetikzlibrary{fit,matrix,positioning}
\usepackage{pdflscape}
\usetikzlibrary{decorations.pathreplacing}
\usepackage{mathrsfs}
\usepackage{makecell}
\usetikzlibrary{decorations.markings}
\usepackage{tikz}
\usetikzlibrary{decorations.markings}
\usetikzlibrary{arrows.meta}
\usetikzlibrary{arrows}
\usepackage{yfonts}
\usepackage{faktor}

\definecolor{darkgreen}{rgb}{0, 0.5, 0}

\newtheorem{theorem}{Theorem}
\newtheorem{lemma}[theorem]{Lemma}

\newtheorem*{theorem*}{Theorem}

\newtheorem{innercustomthm}{Theorem}
\newenvironment{customthm}[1]
  {\renewcommand\theinnercustomthm{#1}\innercustomthm}
  {\endinnercustomthm}

\newcommand\blfootnote[1]{%
  \begingroup
  \renewcommand\thefootnote{}\footnote{#1}%
  \addtocounter{footnote}{-1}%
  \endgroup
}

\newcommand{\ind}{{\rm ind \hspace{.1cm}}}

\newcommand{\C}{\mathbb{C}}
\newcommand{\eps}{\varepsilon}

\tikzstyle{vertex}=[circle, draw, inner sep=0pt, minimum size=13pt]

\tikzstyle{svertex}=[circle, draw, inner sep=0pt, minimum size=8pt]

\begin{document}

\title{Classification of contact seaweeds}

\author[*]{Vincent E. Coll, Jr.}
\author[**]{Nicholas Russoniello}

\affil[*]{\small{Dept. of Mathematics, Lehigh University, Bethlehem, PA 18015:  vec208@lehigh.edu}}
\affil[**]{\small{Dept. of Mathematics, College of William \& Mary, Williamsburg, VA 23185:  nrussoniello@wm.edu}}


\maketitle
\begin{abstract}
\noindent
A celebrated result of Gromov ensures the existence of a contact structure on any connected, non-compact, odd dimensional Lie group. In general, such structures are not invariant under left translation.  
The problem of finding which Lie groups admit a left-invariant contact structure
resolves to the question of determining when a Lie algebra $\mathfrak{g}$ is contact; that is, admits a one-form $\varphi\in\mathfrak{g}^*$ such that $\varphi\wedge(d\varphi)^k\neq 0.$   \\

\noindent
In full generality, this remains an open question; however we settle it for the important category of the evocatively named seaweed algebras by showing that an index-one seaweed is contact precisely when it is quasi-reductive.  Seaweeds were introduced by Dergachev and Kirillov who initiated the development of their index theory -- since completed by Joseph, Panyushev, Yakimova, and Coll, among others.
Recall that a contact Lie algebra has index one -- but not characteristically so. 
Leveraging recent work of Panyushev, Baur, Moreau, Duflo, Khalgui, Torasso, Yakimova, and Ammari, who collectively classified quasi-reductive seaweeds, our equivalence yields a full classification of contact seaweeds. We remark that since type-A and type-C seaweeds are de facto quasi-reductive (by a result of Panyushev), in these types index one alone suffices to ensure the existence of a contact form.

\noindent

\end{abstract}



\blfootnote{Vincent E. Coll, Jr.,  vec208@lehigh.edu, is the corresponding author.
}


\section{Introduction}

This short note, the coda of a sequence of three papers, provides, per the following theorem,  a complete classification of contact seaweed (or biparabolic) Lie algebras.


\begin{customthm}{\ref{thm:main}}
An index-one seaweed is contact if and only if it is quasi-reductive.
\end{customthm}

The first two papers of this sequence establish Theorem \ref{thm:main} for type-A (Coll et al. \textbf{\cite{contactA}}, 2022)  and type-C (Coll and Russoniello \textbf{\cite{contactC}}, 2022) seaweeds by showing, in these classical cases, that index one is sufficient to identify that the seaweed is contact. This is unexpected, as a contact Lie algebra has index one -- but not characteristically so.   These prefatory papers use constructive 
combinatorial methods to develop explicit contact forms for index-one seaweeds.  Panyushev's result (\textbf{\cite{PanRais}}, 2005) that seaweeds in type-A and type-C are quasi-reductive provided a clue to the final form of the above theorem.  And while the comprehensive Theorem \ref{thm:main} subsumes the results of the first two papers, the methods employed here to establish the classification are algebraic and existential, rather than constructive. More specifically, we leverage recent results presented in a series of articles by Ammari (\textbf{\cite{sqparex}}, 2014; \textbf{\cite{sqpar}}, 2014; and \textbf{\cite{sq}}, 2022), which establish that the ``stability'' of a seaweed algebra implies its quasi-reductivity.  This yields the proof heuristic for Theorem \ref{thm:main} in the following diagrammatic form.

$$
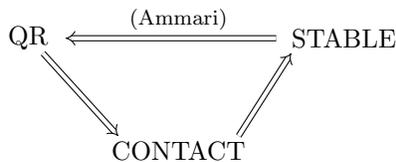
\begin{tikzpicture}
    \node at (0,0) {CONTACT};
    \node at (-2,1.5) {QR};
    \node at (2.2,1.5) {STABLE};
    
    \node at (0,1.75) {\footnotesize(Ammari)};
    
    \draw[-{Implies},double equal sign distance] (-1.8,1.3)--(-0.8,0.2);
    \draw[-{Implies},double equal sign distance] (0.8,0.2)--(1.5,1.3);
    \draw[-{Implies},double equal sign distance] (1.3,1.5)--(-1.5,1.5);
    
    
\end{tikzpicture}$$

\noindent
\textbf{Remark.}  At first blush, Theorem~\ref{thm:main} appears to be a bare equivalence of the notions of contact and quasi-reductivity; however, the collective works of Panyushev \textbf{\cite{PanRais}}, Baur and Moreau (\textbf{\cite{BM}}, 2011), Duflo et al. (\textbf{\cite{duflo}}, 2012), Panyushev and Yakimova (\textbf{\cite{PY}}, 2018), and Ammari \textbf{\cite{sq}} have yielded a complete, explicit classification of quasi-reductive seaweeds. It follows that Theorem~\ref{thm:main} provides a concrete classification in the standard meaning of the term.



\medskip

The structure of this note is as follows.
In Section 2 we acknowledge some preliminary technical results before moving on to the proof of Theorem~\ref{thm:main} in Section 3.  Section 4 is a brief Epilogue which provides a broader context for Theorem~\ref{thm:main}. We are thus able to posit several interesting follow-up questions which merit ongoing investigation.  


\section{Preliminaries}\label{sec:prelim}

\textit{Convention}: All Lie algebras ($\mathfrak{g}, [-, -])$ are over $\C$.
\medskip

A $(2k+1)-$dimensional Lie algebra $\mathfrak{g}$ is called \textit{contact} if it admits a \textit{contact form}; that is, a one-form $\varphi\in\mathfrak{g}^*$ such that $\varphi\wedge(d\varphi)^k\neq 0.$ If $G$ is an underlying Lie group of $\mathfrak{g}$ then each such $\varphi$ gives rise to a left-invariant contact structure on $G,$ and $\varphi\wedge(d\varphi)^k$ is a volume form on $G$.  A now-classical result of Gromov (\textbf{\cite{Gromov}}, 1969) ensures the existence of a contact structure on any connected, non-compact,\footnote{The groups underlying the Lie algebras here considered are non-compact, since all compact complex Lie groups are necessarily abelian.} odd-dimensional Lie group. But, in general, such structures are not invariant under left translations of the Lie group. The problem of finding which Lie groups admit a left-invariant contact structure resolves to cataloging contact Lie algebras and remains a problem of ongoing and topical interest (see \textbf{\cite{defcontact}}, \textbf{\cite{ht2}}, \textbf{\cite{Diatta}}, \textbf{\cite{GR}}, \textbf{\cite{ContactToral}}, and \textbf{\cite{InvContact}}).  
Here, we offer a complete classification for seaweed algebras.


Recall that a \textit{seaweed algebra}, or simply ``seaweed,'' is a subalgebra of a reductive Lie algebra $\mathfrak{r}$ defined by the intersection of two parabolic subalgebras of $\mathfrak{r}$ whose sum is $\mathfrak{r}$.\footnote{This basis-free definition of seaweed has prompted others to refer to these algebras as \textit{biparabolic}. The evocative ``seaweed'' derives from the wavy shape displayed when these algebras are represented in their matrix form in the classical types.}  Seaweeds were first introduced by Dergachev and Kirillov (\textbf{\cite{DK}}, 2000) where the authors studied the ``index" of such algebras in the case where $\mathfrak{r}=\mathfrak{gl}(n).$
Formally, the \textit{index} of a  Lie algebra $\mathfrak{g}$ is an algebraic invariant first introduced by Dixmier (\textbf{\cite{D}}, 1974) and defined as follows: $$\ind\mathfrak{g}=\min_{\varphi\in\mathfrak{g}^*}\dim(\ker(B_{\varphi})),$$ where $B_{\varphi}$ is the skew-symmetric \textit{Kirillov form} given by $B_{\varphi}(x,y)=-d\varphi(x,y)=\varphi([x,y]),$ for all $x,y\in\mathfrak{g}.$ Contact algebras necessarily have index one.



To round out the needed notions for our study, we require  the concepts of ``quasi-reductivity" and ``stability" of a Lie algebra.  
If $\mathfrak{g}$ is an algebraic Lie algebra (in particular, a seaweed) with center $Z(\mathfrak{g}),$ then $\mathfrak{g}$ is \textit{quasi-reductive} if there exists a one-form $\varphi\in\mathfrak{g}^*$ such that $\faktor{\ker(B_{\varphi})}{Z(\mathfrak{g})}$ is a reductive Lie algebra whose center consists of semisimple elements of $\mathfrak{g}.$ Such a one-form $\varphi$ is called a \textit{form of reductive type.} Quasi-reductive Lie groups were first introduced in (Duflo \textbf{\cite{Duflo1982}}, 1982), where it is shown, in particular, that the coadjoint orbits of forms of reductive type can be associated to square integrable unitary representations of the group (see \textbf{\cite{anh}}, \textbf{\cite{duflo}}, and \textbf{\cite{MY}}). Recent work has led to the classification of quasi-reductive seaweeds (see \textbf{\cite{sq}},\textbf{\cite{BM}},\textbf{\cite{duflo}},\textbf{\cite{PanRais}}, and \textbf{\cite{PY}}). Of interest here is a result of Duflo et al. (\textbf{\cite{duflo}}, Th\'{e}or\`{e}me 3.4.1) which establishes, in particular, that all quasi-reductive seaweeds admit a \textit{stable one-form}, i.e., a one-form $\varphi\in\mathfrak{g}^*$ for which there is a Zariski open neighborhood $V\ni\varphi$ such that, for all $\psi\in V,$ the kernels $\ker(B_{\varphi})$ and $\ker(B_{\psi})$ are conjugate under the algebraic adjoint group of $\mathfrak{g}.$ As we will consider stable one-forms in the proof of Theorem~\ref{thm:main}, we provide the following useful equivalence.


\begin{lemma}[Tauvel and Yu \textbf{\cite{TY}}, Th\'{e}or\`{e}me 1.7, 2004; cf. Ammari \textbf{\cite{sqpar}}, 2014]\label{lem:kerstable}
Let $\mathfrak{g}$ be algebraic. A one-form $\varphi$ is stable if and only if $$[\ker(B_{\varphi}),\mathfrak{g}]\cap\ker(B_{\varphi})=\{0\}.$$
\end{lemma}

\noindent We call a Lie algebra that admits a stable one-form a \textit{stable Lie algebra}, and so by Duflo et al., each quasi-reductive seaweed is stable. It remained an open question as to whether all stable seaweeds were quasi-reductive until Ammari showed -- in a series of articles \textbf{\cite{sq}}, \textbf{\cite{sqparex}}, and \textbf{\cite{sqpar}} -- that, for seaweeds, the two notions are commensurate.

\begin{theorem}[Ammari \textbf{\cite{sq}}, Th\'{e}or\`{e}me 5.4; \textbf{\cite{sqparex}}, Propositions 4.22, 4.24, 5.27, and 5.28; \textbf{\cite{sqpar}}, Th\'{e}or\`{e}me 4.2.2]\label{thm:sq}
\centering
A seaweed is quasi-reductive if and only if it is stable.
\end{theorem}

With the assemblage of Lemma~\ref{lem:kerstable} and Theorem~\ref{thm:sq}, we have the groundwork in place to proceed to the algebraic classification of contact seaweeds in the next section.

\medskip


\noindent

\section{Proof of the Main Result}\label{sec:algclass}
In this section, we provide a proof of Theorem~\ref{thm:main}. We first establish some terminology and two technical lemmas, the first of which is a classical characterization of contact Lie algebras.

\begin{lemma}\label{lem:kernel}
Let $\mathfrak{g}$ be an arbitrary Lie algebra. If $\ind\mathfrak{g}=1,$ then a regular one-form $\varphi\in\mathfrak{g}^*$ with $\ker(B_{\varphi})=\mathbb{C}x$ is contact if and only if $\varphi(x)\neq 0.$
\end{lemma}
\begin{proof}
See Lemma 23 in \textbf{\cite{contactA}}.
\end{proof}

Lemma~\ref{lem:kernel} establishes that the existence of a contact form on an index-one Lie algebra $\mathfrak{g}$ 
is equivalent to the existence of a \textit{Reeb vector} associated to a regular one-form $\varphi$; that is, a regular $\varphi\in\mathfrak{g}^*$ is contact precisely when there is some $x_{\varphi}\in\mathfrak{g}$ such that $\ker(B_{\varphi})=\mathbb{C}x_{\varphi}$ and $\varphi(x_{\varphi})=1.$ The identification of the Reeb vector yields a useful vector space decomposition of a contact Lie algebra $\mathfrak{g}=\mathbb{C}x_{\varphi}\oplus \ker(\varphi).$ The subspace $\ker(\varphi)\subset\mathfrak{g}$ defines the contact structure on an underlying Lie group $G$ of $\mathfrak{g}$ (see \textbf{\cite{Geiges}}), and such a decomposition gives rise to a particular \textit{contact basis} of $\mathfrak{g}$ with respect to $\varphi,$ which is explicitly described by the following lemma.

\begin{lemma}[Goze and Remm \textbf{\cite{GR}}, 2014; cf. Alvarez et al. \textbf{\cite{defcontact}}, 2021]\label{lem:contactbasis}
Let $\mathfrak{g}$ be a $(2k+1)-$dimensional contact Lie algebra with contact form $\varphi\in\mathfrak{g}^*.$ Then there exists a basis $\{E_1^*,E_2^*,\dots,E_{2k+1}^*\}$ of $\mathfrak{g}^*$ such that $E_1^*=\varphi$ and $$B_{\varphi}=E_2^*\wedge E_3^*+\dots +E_{2k}^*\wedge E_{2k+1}^*.$$ Moreover, if $\{E_1,E_2,\dots,E_{2k+1}\}$ is the dual basis of $\{E_1^*,E_2^*,\dots,E_{2k+1}^*\},$ we have that
\begin{align*}
    [E_{2\ell},E_{2\ell+1}]&=E_1+\sum_{r=2}^{2k+1}c_{2\ell,2\ell+1}^rE_r,~~~~~~~~~~ \ell=1,\dots,k,\text{ and }\\
    [E_i,E_j]&=\sum_{r=2}^{2k+1}c_{i,j}^rE_r,~~~~~~~~ 1\leq i,j\leq 2k+1 \text{ such that } (i,j)\neq (2\ell,2\ell+1),
\end{align*}
where $c_{i,j}^r\in\mathbb{R}$ for all $i,j,r=2,\dots,2k+1.$
\end{lemma}

We are now in a position to prove the main result of this article.

\begin{theorem}\label{thm:main}
An index-one seaweed is contact if and only if is quasi-reductive.
\end{theorem}
\begin{proof}
We begin with the reverse implication (QR$\implies$contact). Let $\mathfrak{s}$ be an index-one, quasi-reductive seaweed. If $\mathfrak{s}$ has nonzero center, then clearly $\mathfrak{s}$ is contact, so assume $Z(\mathfrak{s})=0.$ Let $\varphi\in\mathfrak{s}^*$ be a regular form of reductive type, and let $h$ be a semisimple element of $\mathfrak{s}$ that spans $\ker(B_{\varphi}).$ If $\varphi(h)\neq 0,$ then $\mathfrak{s}$ is contact by Lemma~\ref{lem:kernel}.

If $\varphi(h)=0,$ then construct a Cartan subalgebra $\mathfrak{h}\subset\mathfrak{s}$ that contains $h,$ and let $\mathscr{B}_h$ be a Cartan-Weyl basis for $\mathfrak{s}$ generated from $h$ (see Coll et al. \textbf{\cite{BCD}}, Lemma 55). Since, for all $x\in\mathfrak{s},$  $\dim(\ker(B_{\varphi}))$ is upper semicontinuous with respect to $\varphi(x),$ there exists $\eps>0$ and a one-form $\psi\in\mathfrak{s}^*$ such that if $\psi(x)=\varphi(x)$ for all $x\in\mathscr{B}_{\mathfrak{h}}\setminus\{h\}$ and $|\psi(h)|<2\eps,$ then $$\ind\mathfrak{g}=1\leq\dim\ker(B_{\psi})\leq\dim\ker(B_{\varphi})=1;$$ that is, if $\mathscr{B}_h^*$ is the basis of $\mathfrak{s}^*$ dual to $\mathscr{B}_h,$ then we may choose an $\eps>0$ such that $\psi=\varphi+\eps h^*$ is regular on $\mathfrak{s}.$ Fix such a $\psi,$ and note that
$$\ker(B_{\psi})=\ker(B_{\varphi})=\mathbb{C}h$$
since $\mathscr{B}_h$ is a Cartan-Weyl basis. Therefore, $\psi$ is a regular one-form on $\mathfrak{s}$ such that $\ker(B_{\psi})=\mathbb{C}h$ and $\psi(h)=\eps\neq 0,$ so $\mathfrak{s}$ is contact by Lemma~\ref{lem:kernel}.

For the forward implication (contact$\implies$QR), assume that $\mathfrak{s}$ is a contact seaweed with contact form $\varphi.$ We claim that $\varphi$ is also a stable one-form. To establish the claim, let $$\mathscr{B}_{\varphi}=\{E_1,E_2,\dots,E_{\dim(\mathfrak{s})}\}$$ be a contact basis of $\mathfrak{s}$ with respect to $\varphi,$ and let $\mathscr{B}^*_{\varphi}$ be the basis of $\mathfrak{s}^*$ dual to $\mathscr{B}_{\varphi}.$ Then $\varphi=E_1^*$ and \begin{equation}\label{eqn:eval}
\varphi([E_1,E_i])=E_1^*\left(\sum_{r=2}^{\dim\mathfrak{s}}c_{1,i}^rE_r\right)=0,
\end{equation}where $c_{i,j}^r\in\mathbb{R}$ for all $i=1,\dots,\dim(\mathfrak{s});$ that is, $\ker(B_{\varphi})=\mathbb{C}E_1.$ Therefore, by Equation~(\ref{eqn:eval}), $$[\ker(B_{\varphi}),\mathfrak{s}]\cap\ker(B_{\varphi})=\{0\},$$ so $\varphi$ is stable by Lemma~\ref{lem:kerstable}. An application of Theorem~\ref{thm:sq} completes the proof.
\end{proof}

\section{Epilogue}

A \textit{Lie proset algebra} $\mathfrak{P}$ is a subalgebra of a reductive Lie algebra $\mathfrak{r}$ that contains a Cartan subalgebra of $\mathfrak{r}.$\footnote{The moniker ``proset'' derives from the fact that, in the classical types, the matrix forms of these algebras naturally correspond to pre-ordered sets (see \textbf{\cite{ContactToral}}).} The category of Lie proset algebras contains seaweed algebras and Lie poset algebras\footnote{If $\mathfrak{r}$ is a reductive Lie algebra with a Cartan subalgebra $\mathfrak{h}$ and a corresponding Borel subalgebra $\mathfrak{b}$, then a Lie poset algebra $\mathfrak{g}$ is any algebra such that $\mathfrak{h}\subset\mathfrak{g}\subset\mathfrak{b}$.  These algebras are Lie semidirect products and are naturally defined by a partial ordering on a set, where the relations are encoded in a Hasse diagram (see \textbf{\cite{CG}}).} as subcategories. In fact, the proof of Theorem~\ref{thm:main} shows that if $\mathfrak{P}$ is an index-one, quasi-reductive Lie proset algebra, then $\mathfrak{P}$ is contact -- and so stable. However, we do not yet know whether the notions of quasi-reductivity and stability are equivalent for Lie proset algebras, and so we are left with the question of whether or not every contact Lie proset algebra is quasi-reductive. One approach to addressing this question is to show that a contact Lie proset algebra admits a contact form with a semisimple Reeb vector. This suffices since for a centerless, contact Lie proset algebra $\mathfrak{P}$ with contact form $\varphi$, $\faktor{\ker(B_{\varphi})}{Z(\mathfrak{P})}=\ker(B_{\varphi})=\mathbb{C}x_{\varphi},$ where $x_{\varphi}$ is the Reeb vector corresponding to $\varphi.$ Therefore, if $x_{\varphi}$ is a semisimple element of $\mathfrak{P},$ then $\varphi$ is a form of reductive type. Such an approach is reliant on a bevy of contact Lie proset algebras; however, such examples are difficult to come by, mainly due to lack 
of a comprehensive computational index theory which specializes to the (rather complete) index theory of seaweeds (see \textbf{\cite{ACam}}, \textbf{\cite{DK}}, and \textbf{\cite{Joseph2006}}) and the nascent (though substantive) index theory of Lie poset algebras (see \textbf{\cite{seriesA}} and \textbf{\cite{BCD}}).

\end{document}